\documentclass{amsart}
\usepackage{mathabx}
\usepackage{comment}
\usepackage{extpfeil}
\usepackage{hyperref}
\usepackage{amssymb, fancyhdr, stmaryrd, amsthm, mathrsfs, textcomp}
\usepackage[utf8]{inputenc}
\usepackage{amsfonts,amsmath,amscd,amssymb}
\usepackage{dsfont,mathtools}
\usepackage{xspace}
\usepackage{biblatex}
\renewbibmacro{in:}{}
\usepackage{xcolor}
\usepackage{tikz-cd}
\usepackage{caption}
\usepackage{subcaption}
\usepackage{graphicx}
\usepackage{listings}
\usepackage{faktor}
\usepackage{calligra}
\usepackage{pifont}

\usepackage{enumitem}
\usepackage{url}

\def\bF{\ensuremath\mathbb{F}}

\def\bH{\ensuremath\mathbb{H}}
\def\bI{\ensuremath\mathbb{I}}
\def\bK{\ensuremath\mathbb{K}}
\def\bN{\ensuremath\mathbb{N}}
\def\bO{\ensuremath\mathbb{O}}
\def\bQ{\ensuremath\mathbb{Q}}

\def\bZ{\ensuremath\mathbb{Z}}

\def\to{\ensuremath\rightarrow}

\def\<{\ensuremath\langle}
\def\>{\ensuremath\rangle}

\DeclareMathOperator{\Irr}{Irr}

\newtheorem{thm}{Theorem}[section]
\newtheorem{mthm}{Theorem}
\newtheorem{lemma}[thm]{Lemma}
\newtheorem{prop}[thm]{Proposition}
\newtheorem{mprop}[mthm]{Proposition}
\newtheorem{cor}[thm]{Corollary}

\theoremstyle{definition}

\newtheorem{mconj}[mthm]{Conjecture}

\newtheorem{mprob}[mthm]{Problem}
\newtheorem{ex}[thm]{Example}
\newtheorem*{ex*}{Example}

\newcommand{\sO}{\mathcal{O}}

\newcommand{\Zp}{{\bZ}_{(p)}}
\newcommand{\ZG}{{\Zp}G}

\begin{document}
\title[A question by Roggenkamp]{
On a question by Roggenkamp about group algebras}
\author{Dylan Johnston}
\author{Dmitriy Rumynin}
\date{\today}
\address{Department of Mathematics, University of Warwick, Coventry, CV4 7AL, UK}

\subjclass{16S34, 16U40, 20C20, 20C15}
\keywords{Krull-Schmidt, semiperfect ring, group algebra, idempotent, finite group}

\begin{abstract}
We investigate whether the group algebra of a finite group over a localisation of the integers is semiperfect. The main result is a necessary and sufficient arithmetic criterion in the ordinary case. In the modular case, we propose a conjecture, which extends the criterion. 
\end{abstract}

\maketitle


\section*{Introduction}

In the present article we would like to contemplate the following problem, posed by Roggenkamp:
\begin{mprob} \label{Pr1} \cite[Kourovka Pr. 4.55]{Kourovka}
 Let $G$ be a finite group and $\Zp$ the localization at $p$. Does the Krull–Schmidt theorem hold for projective $\ZG$-modules?
\end{mprob} 

The first point to clarify is what is meant by the Krull-Schmidt theorem here. Note that in the literature the names of Azumaya and Remak are often attached to it as well, while it is normally formulated for the category of finite length modules with two slightly different, yet equivalent for finite length modules statements, cf. \cite{Berr_Keat,curtis1981methodsVol1}
vs.
\cite{Azu,Facch}. We can interpret Problem~\ref{Pr1} as three different questions about $\ZG$: 

\begin{mprob} \label{Pr2}
Is the category of finitely generated projective $\ZG$-modules a Krull-Schmidt category?
\end{mprob} 

\begin{mprob} \label{Pr3}
Can every projective $\ZG$-module $P$ be written as a direct sum of indecomposable projective modules
$P= \oplus_{i\in I} P_i$
in a unique, up to a permutation and isomorphisms way? 
\end{mprob} 

\begin{mprob} \label{Pr4}
Can every finitely generated projective $\ZG$-module $P$ be written as a direct sum of indecomposable projective modules
\begin{equation} \label{dir_sum}
P=P_1 \oplus \ldots \oplus P_n
\end{equation}
in a unique, up to a permutation and isomorphisms way? 
\end{mprob} 
These problems are interdependent:
\[
\mbox{Yes to Pr. 2}
\Longrightarrow
\mbox{Yes to Pr. 3}
\Longrightarrow
\mbox{Yes to Pr. 4}
\]

In the present paper we address only Problem~\ref{Pr2}. 
The short answer to it is {\bfseries{no}}. 
The reader can benefit from reading the expository paper by Krause \cite{Krause}. The following well-known fact contains the definitions of a semiperfect ring and a Krull-Schmidt category.
\begin{mprop} \label{KRS_cond_gen}
\cite[Prop. 4.1]{Krause}, 
\cite[Th. 2.1]{Bass},
\cite[Ch. 7]{Berr_Keat}
Let $R$ be a ring, $J=J(R)$ its Jacobson radical. The following conditions are equivalent:  
\begin{itemize}
    \item $R$ is semiperfect.
    \item The category of finitely generated projective left $R$-modules is a Krull-Schmidt category.
    \item Every finitely generated projective left $R$-module $P$ admits a decomposition~\eqref{dir_sum}
where each $P_i$ has a local endomorphism ring.
    \item The regular module ${}_RR$ admits a decomposition~\eqref{dir_sum}
where each $P_i$ has a local endomorphism ring.
    \item Every simple left $R$-module admits a projective cover.
    \item Every finitely generated left $R$-module admits a projective cover.
    \item The ring $R/J$ is semisimple artinian and every idempotent in $R/J$ can be lifted to $R$. 
\end{itemize}
\end{mprop}
Now we can explain the negative answer to Problem~\ref{Pr2}, following Woods \cite{Woods}.
Let $G = C_3$, the cyclic group of order $3$, $\bF_q$ -- the finite field of size $q$. 
Observe that 
\begin{equation} \label{QC3}
\mathbb{Q}C_3 \cong \bQ \oplus \bQ(\sqrt[3]{1}), \quad
\mathbb{F}_pC_3 \cong 
\begin{cases}
\mathbb{F}_p \oplus \mathbb{F}_p \oplus \mathbb{F}_p \, ,& \ \mbox{ if } p \equiv 1 \mod 3,\\
\mathbb{F}_p \oplus \bF_{p^2} \, ,& \ \mbox{ if } p \equiv 2 \mod 3.
\end{cases}
\end{equation}
If $p \equiv 1 \mod 3$, $\bF_p G$ has 3 primitive idempotents: too many to be lifted, so that $\ZG$ is not semiperfect.
On the other hand, if 
$p \equiv 2 \mod 3$, both primitive idempotents can be lifted to $\ZG$, making it semiperfect:
\[
\frac{1}{3} (e + x +x^2), \ 
\frac{1}{3} (2e  - x - x^2), 
\ \mbox{ where } \ G = \langle x \rangle \, .
\]

Yet we are not satisfied with this brief negative answer. 
Consequently we 
reinterpret the original question:
\begin{mprob} \label{prob2}
  Let $G$ be a finite group and $\Zp$ the localisation at a prime $p$. 
Find necessary and sufficient conditions on $p$ and $G$ for the ring $\ZG$ to be semiperfect.  
\end{mprob} 

Let us summarize quickly what is already known about it
\begin{mprop} \label{KRS_facts}
The following statements hold 
for a finite group $G$.
\begin{enumerate}
\newcounter{PPS}
    \item If $N\unlhd G$ and $\ZG$ is semiperfect, then $\Zp G/N$ is semiperfect \cite[Lemma 2.2]{Bass}.
    \item $\ZG$ is semiperfect if and only if all idempotents of $\bF_p G$ can be lifted to $\ZG$ \cite[Prop. 8]{Burg}.
    \item Suppose $N\unlhd G$, $H\leq G$, $HN=G$ and $N$ is a $p$-group.
    If $\Zp H$ is semiperfect, then $\ZG$ is semiperfect \cite[Prop. 4.2]{Woods}.
       \item 
    $\Zp S_3$ is semiperfect for every $p$
     \cite[Lemma 6.1]{Woods}.
     \item $\ZG$ is semiperfect if and only $\ZG$ is clean (follows from \cite[Th. 1.1]{Immo}).
      \setcounter{PPS}{\theenumi}  
    \end{enumerate}
Assume further that $G$ is an abelian group of exponent $n=mp^a$, $p \nmid m$ with a Sylow $p$-subgroup $P$.    
    \begin{enumerate}
   \setcounter{enumi}{\thePPS} 
    \item $\ZG$ is semiperfect if and only if $\Zp G/P$ is semiperfect \cite[Lemma 5.1]{Woods}.
    \item $\ZG$ is semiperfect if and only if $\Zp C_n$ is semiperfect \cite[Prop. 5.7]{Woods}.
    \item 
    $\ZG$ is semiperfect if and only if 
    each monic factor of $z^m-1\in \bF_p [z]$
    can be lifted to $\Zp [z]$  \cite[Th. 5.8]{Woods}.
    \item $\ZG$ is semiperfect if and only $p$ is a primitive root of 1 modulo $d$ for each $d |m$ \cite[Lemma 2.6]{Immo}.
\end{enumerate}
\end{mprop}

Notice further that a complete satisfactory answer to Problem~\ref{prob2} is known for abelian $G$ \cite[Th. 2.11 and Th. 3.5]{Immo}, yet we skip it here as it boils down to getting to the bottom of the arithmetic  of the last statement of Proposition~\ref{KRS_facts}.

For general $G$ and $p$, the ring $\ZG /J$ is semisimple and artinian -- it is even a finite dimensional algebra. So the real issue is lifting of idempotents. The idempotents can be lifted to $\widehat{\Zp}G$, where $\widehat{\Zp}$ is the $p$-adic integers. This is a classical well-known topic \cite{Berr_Keat,curtis1981methodsVol1}: 
$\widehat{\Zp}G$ is always semiperfect.
Here, in this paper, we are trying to grapple with 
the non-complete local ring $\Zp$.

We can give a fully satisfactory answer to Problem~\ref{prob2} only in the ``ordinary'' case. The answer can be given in terms of the standard arithmetic information attached to the complex irreducible characters $\chi\in \Irr(G)$. 
\begin{mthm} \label{Main_Result}
Suppose $p$ doesn't divide $|G|$. 
Then the ring  $\mathbb{Z}_{(p)}G$ is semiperfect if and only if both of the following conditions hold:
\begin{enumerate}[label=\roman*),left=0ex]
\item For each $\chi\in\Irr(G)$ its Schur index $\iota (\chi)$ is $1$.
\item For each $\chi\in\Irr(G)$ the prime $p$ is inert in the field of character values $\bQ(\chi)$
\end{enumerate}    
\end{mthm}

The modular case, when $p\mid |G|$, is harder
to pinpoint. Yet we propose a conjecture
that yields a workable criterion in the modular case.

To state it, we recall some standard notation.

Let $R$ be a ring. Let $\mathbf{G}(R)$ and $\mathbf{K}(R)$ denote the Grothendieck rings over $\bZ$ of finitely generated $R$-modules and of finitely generated projective $R$-modules, respectively. Let $\mathbf{G}^+(R)$ and $\mathbf{K}^+(R)$ be the corresponding positive Grothendieck semigroups, i.e., $\bN$-spans of modules inside the Grothendieck groups. Consider the natural maps
\begin{equation} \label{KGmaps}
\mathbf{G}(\bQ G)
\xrightarrow{\alpha}
\mathbf{G}(\bF_p G)
\xleftarrow{\beta}
\mathbf{K}(\bF_p G)
\end{equation}
where $\alpha$ is the decomposition map \cite[Def. 16.20]{curtis1981methodsVol1}
and $\beta$ is the Cartan homomorphism 
\cite[Def. 18.4]{curtis1981methodsVol1}.

\begin{mconj} \label{KGconj}
The ring $\mathbb{Z}_{(p)}G$ is semiperfect if and only if 
$\alpha (\mathbf{G}^+(\bQ G))
\supseteq
\beta (\mathbf{K}^+(\bF_p G))$.
\end{mconj}

In the first section we set up preliminaries on liftings of modules, our approach to Problem~\ref{prob2}.
In the second section we consider the ordinary case, i.e., $p\nmid |G|$, in particular, we prove Theorem~\ref{Main_Result}.
The final section is devoted to the general case and contains a discussion of Conjecture~\ref{KGconj}.

\section*{Acknowledgements}
The first author was supported by the Heilbronn Institute for Mathematical Research (HIMR), grant EP/V521917/1. The second author is grateful to the Department of Mathematics of the University of Zurich for its hospitality.

\section{Preliminaries}
Let $\bO$ be a discrete valuation ring, $\bI=(\varpi)$ -- its maximal ideal with uniformiser $\varpi$. We also consider its residue field $\bF \coloneqq \bO/\bI$ and its fraction field $\bK \coloneqq Q(\bO)$.

Let $A$ be an $\bO$-order, that is, an associative $\bO$-algebra, finitely generated projective as an $\bO$-module. Given a ring homomorphism $\phi: \bO \to R$, by $A_\phi$ or $A_R$, when $\phi$ is clear by default, we denote the specialisation $R \otimes_{\bO} A$. 
The two key specialisations come from the two points of $\mbox{Spec}(\bO)$:
the algebras $A_\bK$ and $A_\bF$ are finite-dimensional of the same dimension.

\begin{lemma} \label{Jacobson}
The restriction to the Jacobson radicals $\psi: J(A) \to J(A_\bF)$ of the natural homomorphism $\psi: A \to A_\bF$ is surjective. 
Furthermore, $J(A) = \psi^{-1} (J(A_\bF))$.    
\end{lemma}
\begin{proof}
Since $\psi$ is surjective, $\psi (J(A)) \subseteq J(A_\bF)$ and $J(A) \subseteq \psi^{-1} (J(A_\bF))$.    

Now consider a simple $A$-module $M$. The multiplication by $\varpi$ is a homomorphism
$M\to M$. By Schur's Lemma, it is either zero or an automorphism. Note further that it cannot be an automorphism. Suppose it is. Then $M=\bK \otimes_{\bO}M$ is already an $A_\bK$-module.
Choose finitely many generators $a_1, \ldots, a_n$ of the $\bO$-module $A$ and a nonzero 
$m\in M$. Let $M' \subseteq M$ be the $\bO$-submodule generated by $a_1m, \ldots, a_nm$. Clearly, $M'$ is an $A$-submodule. By simplicity, $M'=M$, which is impossible for a $\bK$-vector space.

Hence, $\varpi M =0$ and $M$ is an $A_\bF$-module. It follows that 
$J(A)$ and $J(A_\bF)$ consist of the annihilators of the same set of modules. 
Consequently, $\psi: J(A) \to J(A_\bF)$ is surjective and 
$J(A) = \psi^{-1} (J(A_\bF))$.  
\end{proof}

\begin{cor} \label{cor1}
The rings $A/J(A)$ and $A_\bF/J(A_\bF)$ are isomorphic.    
\end{cor}

We need to talk about lifting of modules. Let $M_\bF$ be an $A_\bF$-module. We say that an $A$-module $M$ is {\em its lifting} if $M$ is a projective $\bO$-module 
and $M_\bF \cong \bF \otimes_\bO M$ as $A_\bF$-modules. The following observation is useful:
\begin{lemma} \label{p_lift}
\cite[Th. 30.11]{curtis1981methodsVol1}   
Suppose $M$ is a lifting of $M_\bF$.
The $A$-module $M$ is projective if and only if the $A_\bF$-module $M_\bF$ is projective.
\end{lemma}
We are ready for our main criterion, a handy supplement to Proposition~\ref{KRS_cond_gen}.

\begin{prop} \label{KRS_cond}
The following conditions for the $\bO$-order $A$ are equivalent.
\begin{enumerate}
    \item $A$ is  semiperfect.
    \item Every idempotent of $A_\bF$ can be lifted to an idempotent of $A$.
    \item Every principal indecomposable $A_\bF$-module admits a 
    lifting.
    \item Every finitely generated projective $A_\bF$-module admits a 
    lifting.
\end{enumerate}    
\end{prop}
\begin{proof}
The equivalence $(1)\Leftrightarrow (2)$ follows from 
Corollary~\ref{cor1}. Indeed, it ensures that $A/J(A)$ is a finite-dimensional semisimple algebra. 
By Proposition~\ref{KRS_cond_gen}, $A$ is semiperfect if and only if all idempotents of $A/J(A)\cong A_\bF/J(A_\bF)$ can be lifted to $A$. 
It remains to compare idempotents in three algebras
\[ A \xrightarrow{\psi} A_\bF \xrightarrow{\pi} A/J(A)\cong A_\bF/J(A_\bF)\, .\]
Indeed, an idempotent $e\in A/J(A)$ can always be lifted
to an idempotent $e_1\in A_\bF$ because $J(A_\bF)$ is nilpotent. Now use $(2)$ to lift it to an idempotent $e_2\in A$. This proves $(1)$.

In the opposite direction, assume $(1)$. 
Pick an idempotent $e\in A_\bF$, lift $\pi(e)$ to an
idempotent $e_1\in A$. This proves $(2)$.

The implication
$(1)\Rightarrow (4)$ is contained within  Proposition~\ref{KRS_cond_gen}.
Indeed, let $P_\bF$ be a finitely generated projective $A_\bF$-module.
As an $A$-module, it is also finitely generated. Proposition~\ref{KRS_cond_gen}
yields us an $A$-projective cover $\pi : P \twoheadrightarrow  P_\bF$.
Note that $A$-projectivity implies $\bO$-projectivity.

Since $\varpi P_\bF =0$, we conclude that $\ker \pi \supseteq \varpi P$.
The induced $A$-module homomorphism
$\pi'$ yields a decomposition
\[\pi: P \xrightarrow{\psi} P/\varpi P \xrightarrow{\pi'} P_\bF\, .\] 
Since $P/\varpi P$ is an $A_\bF$-module,
projectivity of $P_\bF$ implies that $\pi'$ is split:
$P/\varpi P = M \oplus P_\bF$ for another submodule $M$. Hence,
\[
P = \psi^{-1}(M) + \psi^{-1} (P_\bF) \subseteq \ker (\pi) + \psi^{-1} (P_\bF) \, .
\]
Since $\ker (\pi)$ is superfluous,  $P = \psi^{-1} (P_\bF)$ and $M=0$.
Therefore, $\ker \pi = \varpi P$
and $P\otimes_{\bO} \bF \cong P_\bF$, so that $P$ is a 
lifting.

To prove $(3)\Rightarrow (1)$, we construct a projective cover of a simple $A$-module $M$.
By Corollary~\ref{cor1}, $M$ is an $A_\bF$-module. Let $P_\bF (M)$ be its $A_\bF$-projective cover. Since $P_\bF (M)$ is a principal indecomposable, it admits a 
lifting $P$.
We claim that $P$ is an $A$-projective cover of $M$. Indeed, it is projective by Lemma~\ref{p_lift} and we have surjective homomorphisms of $A_\bF$ modules
\[
\pi: P \xrightarrow{\psi} P_{\bF} (M) \xrightarrow{\pi'} M \, .
\]

It remains to argue that $\ker (\pi)$ is superfluous. Assume $L+\ker (\pi)=P$.
It follows that $\psi(L)+\ker (\pi')=P_\bF (M)$. Since $\pi'$ is an $A_\bF$-projective cover,
$\ker (\pi')$ is superfluous and $\psi(L) =P_\bF (M)$.
It follows that $L+ \ker (\psi) = L+ \varpi P = P$.
Observe that $\psi$ is an $\bO$-module projective cover.
Hence, $\ker (\psi)$ is $\bO$-superfluous and $L =P$.

The final implication
$(4)\Rightarrow (3)$ is obvious.
\end{proof}

The final well-known corollary is immediate since completeness of $\bO$ enables lifting of idempotents.
\begin{cor}
If $\bO$ is complete, $A$ is semiperfect.    
\end{cor}

\section{The semisimple case}
Before specialising to the case of the group algebra, we make several general observations in the case, when $A_\bF$ is semisimple. 
The following observation is not fully routine because semisimplicity (in contrast to separability) is not an open condition:
\begin{lemma} \label{semisimple}
If $A_\bF$ is semisimple, then $A_\bK$ is semisimple.    
\end{lemma}
\begin{proof}
Suppose $A_\bK$ is not semisimple. Then $I\coloneqq A \cap J(A_\bK)$ is a non-zero nilpotent ideal of $A$. Since $A_\bF$ is semisimple, it has no nilpotent ideals. Therefore, $\psi(I)=0$
and $I\subseteq \varpi A = \ker \psi$, where $\psi: A\to A_\bF$ is the reduction homomorphism.

Consider now $I' = \varpi^{-1} I \coloneqq \{ x\in A \,\mid\, \varpi x \in I \}$.
Clearly, $I'$ is still a subset of $J(A_\bK)$. Hence, $I'=I$.
Now pick a nonzero element $\varpi^n a\in I$ with $a\in A\setminus \varpi A$ and the smallest possible $n$. Since $I\subseteq \varpi A$, $n\geq 1$. Then 
$\varpi^{n-1} a\in I'$, proving that $I'\neq I$, a contradiction. 
\end{proof}
In light of Lemma~\ref{semisimple}, the next step is to compare Artin-Wedderburn decompositions. Let us fix them:
\begin{equation} \label{ArWe}
A_\bF \cong M_{n_1}(F_1) \oplus \ldots \oplus M_{n_k}(F_k)  , \quad  
A_\bK \cong M_{m_1}(K_1) \oplus \ldots \oplus M_{m_t}(K_t)
\end{equation}
where $F_i$, $K_i$ are finite-dimensional division algebras over $\bF$ and $\bK$, respectively. By $\iota (F_i) \coloneqq \sqrt{\dim_{Z(F_i)} F_i}$,
$\iota (K_i) \coloneqq \sqrt{\dim_{Z(K_i)} K_i}$ we denote their Schur indices.


\begin{prop} \label{KRS_cond_2}
Suppose that $A_\bF$ is semisimple with Artin-Wedderburn decompositions as in~\eqref{ArWe}. In addition to Proposition~\ref{KRS_cond}, the following conditions for the $\bO$-order $A$ are equivalent.
\begin{enumerate}
\newcounter{PPP}
    \item $A$ is semiperfect.
    \item Every simple $A_\bF$-module admits a lifting.
\setcounter{PPP}{\theenumi}
\end{enumerate} 
Let us assume further that $A$ is a maximal order. Then the above conditions are also equivalent to the following two conditions.
\begin{enumerate}
\setcounter{enumi}{\thePPP}
    \item The numbers of simple modules are equal: $k=t$. Furthermore, after a suitable permutation every simple $A_\bK$-module $K_i^{m_i}$ admits an $A$-lattice $M_i$ such that the $A_\bF$-modules $M_i\otimes_\bO \bF$ and $F_i^{n_i}$ are isomorphic. 
    \item The numbers of simple modules are equal: $k=t$. Furthermore, after a suitable permutation for every $i$ we have $n_i=m_i$, $\iota (F_i) = \iota (K_i)$ and every division algebra $K_i$ admits an $\bO$-order $L_i$ such that the $\bF$-algebras $L_i\otimes_\bO \bF$ and $F_i$ are isomorphic. 
\end{enumerate} 
\end{prop}
\begin{proof}
The implication
$(3)\Rightarrow (2)$ is obvious: the $A$-module $M_i$ is an explicit lifting 
of a simple $A_\bF$-module $F_i^{n_i}$.  

The equivalence $(2)\Leftrightarrow (1)$ follows from Proposition~\ref{KRS_cond}.
Since $A_\bF$ is semisimple, its principal indecomposable modules are the same as simple
modules. 
By 
Proposition~\ref{KRS_cond}, they admit liftings.

To prove the implication
$(3)\Rightarrow (4)$, note that $M_i$
is a projective $A$-module by Lemma~\ref{p_lift}.
Define $L_i$ as the endomorphism ring $L_i\coloneqq \mbox{End}_AM_i$. Since $M_i$ is $A$-projective, the natural map
\[
L_i \rightarrow \mbox{End}_{A_\bF} F_i^{n_i} \cong F_i
\]
is surjective and $L_i \otimes_{\bO} \bF \cong F_i$. Observe that the second natural map
\[
\phi: L_i \rightarrow L_i\otimes_{\bO} \bK 
\rightarrow
\mbox{End}_{A_\bK} (M_i\otimes_{\bO} \bK) \cong 
\mbox{End}_{A_\bK} K_i^{m_i} \cong K_i 
\]
is injective.
If $w\in L_i$ and $w\neq 0$,
then $w(x) \neq 0$ for some $x\in M_i\subseteq K_i^{m_i}$.
Hence, $\phi(w)(x) \neq 0$ and $\phi(w)\neq 0$.

Now observe that $\phi(L_i)\cong L_i$ is a lattice in $K_i$. Pick any $w\in K_i$. For some natural $n$ we have $\varpi^n w(M_i) \subseteq M_i$. Thus, $\varpi^n w\in \phi (L_i)$.

The outstanding requirements follow immediately:
$n_i=m_i$ and $\iota (K_i) = \iota (F_i)$ .

The remaining implications utilise the requirement that $A$ is a maximal order.

Let us prove $(1)\Rightarrow (3)$.
Note that a simple $A_\bF$-module $F_i^{n_i}$ has a form $A_\bF e$ for a minimal idempotent $e$ in $A_\bF$.
It can be lifted to a minimal idempotent $\tilde{e}\in A$. Observe that $M_i \coloneqq A\tilde{e}$ is a lifting of $F_i^{n_i}$.  
Since $\tilde{e}$ is minimal, $M_i$ is an indecomposable
$A$-module. By \cite[Th. 26.12]{curtis1981methodsVol1}, 
$A_\bK \tilde{e}$ is a simple $A_\bK$-module.
After a suitable permutation, it is 
$K_i^{m_i}$.  The implication is proved.

The final implication
$(4)\Rightarrow (2)$ is an immediate corollary of the structure theorem for maximal orders. By \cite[Th. 26.20]{curtis1981methodsVol1}, $A = A_1 \oplus \ldots \oplus A_t$ where each $A_i$ is a maximal order in $M_{m_i} (K_i)$.
By \cite[Th. 26.23]{curtis1981methodsVol1}, after a change of basis, we have
$A_i=M_{m_i} (B_i)$ for a maximal order $B_i\leq K_i$, and moreover, $B_i \supseteq L_i$. It follows that the natural map 
\[
F_i \cong L_i \otimes_\bO \bF \rightarrow B_i \otimes_\bO \bF 
\]
is an isomorphism. Hence, 
$A_i \otimes_\bO \bF \cong M_{n_i} (F_i)$
and all simple $A_\bF$-modules have liftings.
\end{proof}

Let us now specialise to the situation of our primary concern:
$A \coloneqq \bO G$ for a finite group $G$,
where
$\bO \coloneqq \mathbb{Z}_{(p)}$, the localisation of the integers at the prime $p$, $p \nmid |G|$. 
Note that
$\mathbb{F} =\mathbb{Z}_{(p)}/p\mathbb{Z}_{(p)} \cong \mathbb{F}_p$
 and  $\bK = \bQ$, the rational numbers.
Observe that $A$ is a maximal order in
$\bQ G$ \cite[Prop. 27.1]{curtis1981methodsVol1}
so that
Proposition~\ref{KRS_cond_2}
is fully applicable.
Let us use it to examine two examples.

\begin{ex}
Let $G = C_n$, $p\nmid n$, a slight generalisation of~\eqref{QC3}. 
The following argument is essentially 
a proof of the last two statements
of Proposition~\ref{KRS_facts}.
Clearly,
\[ \bQ G \cong 
\bigoplus_{d|n} \bQ(\sqrt[d]{1}) 
\cong \bigoplus_{d|n} \dfrac{\bQ [z]}{(\Psi_d(z))}, \quad
\bF_p G \cong 
\bigoplus_{i} \dfrac{\bF_p [z]}{(\psi_i(z))} \] 
where $\Psi_d$ is the $d$-th cyclotomic polynomial and $\psi_i$ are ``cyclotomic polynomials" modulo $p$, i.e., prime factors of $z^n-1 \in \bF_p [z]$. All Schur indices are equal to 1. By Proposition~\ref{KRS_cond_2} and \cite[Th 3.5]{Lidl}, the following three conditions are equivalent:
\begin{itemize}
    \item $\ZG$ is semiperfect,
    \item for each $d\mid n$ the polynomial $\Psi_d (z)$ remains irreducible modulo $p$,
    \item for each $d\mid n$ the order of $p$ in the group $\mbox{GL}_1 (\bZ /(d))$ is $\varphi (d)$ (maximal allowed by Euler's Theorem).
\end{itemize}
\end{ex}
Let us now see an example where the Schur index plays a role.
\begin{ex}
Let $G = Q_8$, the quaternion group of order $8$, $p\neq 2$. Then
\[
\bQ G \cong \bQ^{\oplus 4} \oplus \bH_{\bQ}, \quad \bF_p G \cong {\bF}_p^{\oplus 4} \oplus M_2(\bF_p)
\]
where ${\bH}_{\bQ}$ is the quaternion algebra over $\mathbb{Q}$.
Then $\iota (\bH_{\bQ} )=2 \neq 1 = \iota (\bF_p)$.
By Proposition~\ref{KRS_cond_2}, $\ZG$ is not semiperfect.
\end{ex}

Let us recall the well-known facts about the Artin-Wedderburn decomposition of $\bQ G$
in terms of its complex characters. 
They are a combination of the Brauer-Fein theorem \cite{fossum1974center}
and various standard facts \cite[Section 12.2]{serre1977linear}.

\begin{lemma}
Let $G$ be a finite group with the Artin-Wedderburn decomposition of $A_\bK = \bQ G$ as in \eqref{ArWe}.
Let $Z_i \coloneqq Z(K_i)$,
$s_i = \dim_{\bQ}Z_i$
and
$d_i = \iota (K_i)$.
The following statements hold.

\begin{enumerate}[label=\roman*),left=0.5ex]
\item For each summand $M_{m_i}(K_i)$, 
we have $M_{m_i}(K_i)\otimes_{\bQ} \bar{\bQ} \cong M_{m_id_i}(\bar{\bQ})^{\oplus s_i}$, which yields $s_i$ 
distinct complex characters 
$\chi_{i,1}, \ldots, \chi_{i,s_i}$
of degree $m_id_i$.
\item The set of complex irreducible characters is a disjoint union
\[
\Irr (G) \; = \; 
\coprod_{i=1}^{t} \{\chi_{i,1}, \ldots, \chi_{i,s_i}\} \, . 
\]
\item For each $i$ and $j \in \{1,\ldots, s_i\}$ the field $Z_i$ is $\mathbb{Q}(\chi_{i,j})$, the field of values of the character.
\item For each $i$, the extension $Z_i \supseteq\mathbb{Q}$ is Galois.
\item The characters $\chi_{i,1}, \ldots, \chi_{i,s_i}$ constitute one $\textup{Gal}(Z_i/\mathbb{Q})$-orbit.
\end{enumerate}
\end{lemma}

Now, we are ready to prove 
the main theorem.
\begin{proof} {\em (of Theorem~\ref{Main_Result})}
Recall that for an algebraic number field $\bK$, the ring $\sO_{\bK}/(p)$ of integers modulo $p$ can be either a field, or a direct sum of at least two fields, or a non-reduced ring. These mutually exclusive possibilities mean that the prime $p$ is correspondingly either 
{\em inert}  in $\bK$,
or {\em split}  in $\bK$,
or {\em ramified}  in $\bK$.

Now $A=\ZG$ has Artin-Wedderburn decompositions~\eqref{ArWe}. 
Since $A$ is a maximal order (cf. the proof of $(4)\Rightarrow (2)$ of Proposition~\ref{KRS_cond_2}), we can write
\begin{equation} \label{ArWe2}
\ZG \cong M_{m_1}(B_1) \oplus \ldots \oplus M_{m_t}(B_t)\, ,
\quad  
\bF_p G \cong \oplus_{i=1}^t M_{m_i}(B_i\otimes_{\Zp} \bF_p)    
\end{equation}
where $B_i$ is a maximal order in $K_i$.

By Proposition~\ref{KRS_cond_2}, $A$ is semiperfect if and only $B_i\otimes_{\Zp} \bF\cong F_i$ for all $i$, after a suitable permutation of indices.

We proceed to prove the ``only if" direction.
By the Little Wedderburn Theorem, all $F_i$ are finite fields. This forces commutativity of both $B_i$ and $K_i$. 
Since $K_i = \bQ (\chi_{i,j})$ is commutative for all $i$ and $j$,
we conclude that $\iota (\chi)=1$ for all $\chi\in\Irr (G)$.
Further, we can rewrite the decompositions~\eqref{ArWe2} as
\begin{equation} \label{ArWe3}
\ZG \cong \oplus_{i=1}^t 
M_{m_i}(\sO_{\bQ (\chi_{i,1})} )\, ,
\quad  
\bF_p G \cong \oplus_{i=1}^t M_{m_i}(\sO_{\bQ (\chi_{i,1})} /(p) \, )\, ,
\end{equation}
where Proposition~\ref{KRS_cond_2}
forces all
$\sO_{\bQ (\chi)} \otimes_{\Zp} \bF_p$
to be fields. This means that the prime
$p$ is inert in all $\bQ (\chi)$.

For the ``if" direction, observe that the first condition (all $\iota (\chi)=1$) enables us to write the Artin-Wedderburn decompositions as in~\eqref{ArWe3}.
Now the second condition ensures that all
$\sO_{\bQ (\chi)} \otimes_{\Zp} \bF_p$
are fields. By Proposition~\ref{KRS_cond_2},
$\ZG$ is semiperfect.
\end{proof}

\section{The general case}
Let us consider the general case now.
Note that if $p\nmid |G|$, 
$\ZG$ is not a maximal order \cite[Prop. 27.1]{curtis1981methodsVol1}. This 
creates additional challenges.
However, many standard results still hold
under the assumption that $\ZG$
is semiperfect \cite[Ch. 17]{curtis1981methodsVol1}. In particular, we can settle the only-if direction of Conjecture~\ref{KGconj}.
\begin{prop}\label{clear-direction-of-KGconj}
If the ring  $\mathbb{Z}_{(p)}G$ is semiperfect,
then
$\alpha (\mathbf{G}^+(\bQ G))
\supseteq
\beta (\mathbf{K}^+(\bF_p G))$.   
\end{prop}
\begin{proof}
The reduction modulo $p$ yields an isomorphism 
$\mathbf{K}^+(\ZG)\xrightarrow{\cong}\mathbf{K}^+(\bF_p G)$,
cf. \cite[Th. 18.2]{curtis1981methodsVol1}.
From the commutativity of the Cartan-Brauer Triangle \cite[Prop. 18.5]{curtis1981methodsVol1}, 
it suffices to observe that the natural map
\[
\gamma: \mathbf{K}(\ZG)\rightarrow \mathbf{G}(\bQ G), \quad [M] \mapsto [\bQ\otimes_{\Zp} M]
\]
takes $\mathbf{K}^+(\ZG)$
to $\mathbf{G}^+(\bQ G)$. But this is obvious because $\bQ \otimes_{\Zp} M$ can be wrtten as a finite direct sum of simple $\bQ G$-modules.
\end{proof}

Our second observation is that Conjecture~\ref{KGconj} holds true 
under a condition on the Sylow $p$-subgroup.
Note that this trivially covers the case $p\nmid |G|$. Furthermore, Conjecture~\ref{KGconj} reduces to 
Theorem~\ref{Main_Result} in this case,
and therefore, holds.

\begin{prop}
Suppose the group $G$ is a semidirect product $P\rtimes H$ of its Sylow $p$-subgroup $P$ and a $p'$-subgroup $H$. 
The following statements are equivalent.
\begin{enumerate}
    \item $\alpha (\mathbf{G}^+(\bQ G))
\supseteq
\beta (\mathbf{K}^+(\bF_p G))$.
    \item $\ZG$ is semiperfect.
    \item $\Zp H$ is semiperfect.
\end{enumerate}
\end{prop}
\begin{proof} 
Proposition~\ref{clear-direction-of-KGconj} asserts 
$(2) \Rightarrow (1)$. 

The implication $(3) \Rightarrow (2)$
is part (3) of Proposition~\ref{KRS_facts}.

Let us prove the final implication
$(1) \Rightarrow (3)$.
By $I_R$ we denote the ideal $(1-x)_{x\in P}$ in the group algebra $RG$. Since $P$ is normal,
\[
I_R = \ker (RG\rightarrow RH), \ \ 
R G /I_R \cong R H \ 
\mbox{ and } \ 
I_{\bF_p} = J(\bF_p G) \, . \ 
\]
Observe that for every $\ZG$-module $M$, we have isomorphisms of $\bF_p H$-modules
\[
\bF_p H \otimes_{\bF_p G} (\bF_p \otimes_{\Zp} M)
\cong
\bF_p \otimes_{\Zp} (\Zp H \otimes_{\ZG} M)
\cong 
\bF_p H \otimes_{\ZG} M,
\]
which, in its turn, allows us to conclude that the diagram
\[
\begin{CD}
\mathbf{G}^+(\bQ G) @>>{\alpha_G}> \mathbf{G}^+(\bF_p G)\\
@V{M\mapsto \bQ H \otimes_{\bQ G} M \cong M/(I_{\bQ}M)}V{\gamma_\bQ}V @V{\gamma_{\bF_p}}VV\\
\mathbf{G}^+(\bQ H) @>{\alpha_H}>> \mathbf{G}^+(\bF_p H)
\end{CD}
\]
is commutative. Let us fix the Artin-Wedderburn decompositions for $H$ as in~\eqref{ArWe}, so that we have the sets of simple modules
\[
\Irr (\bF_p G) = 
\Irr (\bF_p H) =
\{ F^{n_i}_i\}_{i=1,\ldots,k}, \ 
\Irr (\bQ H) =
\{ K_i^{m_i}\}_{i=1,\ldots,t}
\subseteq
\Irr (\bQ G),
\]
whose classes constitute $\bZ$-bases of the four Grothendieck groups in the diagram. 
Let $P_i$ be the $\bF_p G$-projective cover of $F_i^{n_i}$. Notice that
\[
\gamma_{\bF_p} ([P_i]) =
\gamma_{\bF_p} ([F_i^{n_i}]) = [F_i^{n_i}]\, .
\]
Condition~(1) tells us that the classes $[P_i]$ belong to the image of $\alpha_G$. Hence, the composition 
$\gamma_{\bF_p}\circ \alpha_G = \alpha_H \circ \gamma_{\bQ}$ is surjective. Now observe that 
\[
L \in \Irr (\bQ G) \ \Longrightarrow \ 
\gamma_{\bQ} ([L]) =
\begin{cases}
[L]  & \mbox{ if } L\in \Irr (\bQ H) ,\\
0 &    \mbox{ if } L\not\in \Irr (\bQ H).
\end{cases}
\]
Hence, $\gamma_{\bQ}$ is surjective, and therefore, $\alpha_H$ is surjective. 
Let 
$L_i$ be a maximal order in $K_i$, $s_i = \iota (K_i)$. 
Semisimplicity of $\bF_p H$ means that $p$ cannot be ramified in the field $Z(K_i)$. Computationally, 
\[
\alpha_H ([K_i^{m_i}]) = 
[(L_i \otimes_{\Zp} \bF_p)^{m_i}]
= s_i 
\sum_{j\in X(i)} [F_j^{n_j}]
\]
where 
the sets of indices $X(i)$ do not intersect for different $i$. Because of that, surjectivity of $\alpha_H$ implies that 
$s_i=1$ 
and each set of indices $X(i)$ consist of a single index. 
Hence, the conditions of Theorem~\ref{Main_Result} are met for $H$
and $\Zp H$ is semiperfect.

\end{proof}

The second condition of Conjecture~\ref{KGconj} essentially means that every projective $\bF_p G$-module $P$ can be written in the $K$-theory as
\[
[P] = \sum_{V\in \Irr (\bQ G)} N_V [\bF_p \otimes_{\Zp} \widetilde{V}] \in
\mathbf{G}(\bF_p G)
\quad \mbox{ with } N_V\in\bN 
\]
where $\widetilde{V}$ are ${\ZG}$-lattices in the $\bQ G$-modules $V$. This condition can be strengthened. We say that a filtration on an $\bF_p G$-module $M$
\[
M = M_k \geq M_{k-1} \geq \ldots \geq M_1 \geq M_0 = 0
\]
is {\em a Brauer-Humphreys filtration} if each consequent quotient $M_{i}/M_{i-1}$ 
is isomorphic to 
$\bF_p \otimes_{\Zp} \widetilde{V}$
where
$\widetilde{V}$ is a $\ZG$-lattice in a simple $\bQ G$-module $V$.
\begin{prop}
If every projective indecomposable $\bF_p G$-module has a Brauer-Humphreys filtration,
then  $\ZG$ is semiperfect.
\end{prop}
\begin{proof} Let $P$ be a projective indecomposable $\bF_p G$-module, $P_i$ -- its Brauer-Humphreys filtration.
We construct the lifting $\widetilde{P}_i$ of each $P_i$ recursively, $i=1,2,\ldots, n$. 

The definition of the filtration yields $\widetilde{P}_1$.

Now suppose $\widetilde{P}_i$ is constructed. Let $\widetilde{V}$ be a lifting of $P_{i+1}/P_i$. By \cite[Cor. 25.6]{curtis1981methodsVol1}, the natural map
\[
\psi = `` \bF_p \otimes_{\bZ} \, - \, ":
\mbox{Ext}^{1}_{\ZG} (\widetilde{P_i},\widetilde{V})
\xrightarrow{\cong}
\mbox{Ext}^{1}_{\bF_p G} ({P_i},{V})
\]
is an isomorphism. A $\ZG$-module representing $\psi^{-1} (P_{i+1})$ is a lifting of $P_{i+1}$.

By Proposition~\ref{KRS_cond}, $\ZG$ is semiperfect. 
\end{proof}

We finish the paper with 
an investigation of the semiperfectness of $\ZG$ for the smallest simple non-abelian group.

\begin{ex}
Let $G=A_5$.
Its complex character table is given in Table \ref{complex chars of A_5}. We use the degrees as the character indices and label the columns by the orders of elements in the associated conjugacy class. We also include columns describing $\mathbb{Q}(\rho)$ and the Schur index.
\begin{table}[h!]
\centering
$\begin{array}{c|ccccc|cc}
  \rm &1&2&3&5A&5B&\mathbb{Q}(\rho)&\iota(\rho) \\
\hline
  \rho_{1}&1&1&1&1&1&\mathbb{Q}&1\\
  \rho_{3}&3&-1&0&\varphi&1-\varphi&\mathbb{Q}(\sqrt{5})&1\\
  \widetilde{\rho_{3}}&3&-1&0&1-\varphi&\varphi&\mathbb{Q}(\sqrt{5})&1\\
  \rho_{4}&4&0&1&-1&-1&\mathbb{Q}&1\\
  \rho_{5}&5&1&-1&0&0&\mathbb{Q}&1\\
\end{array}$
\caption{Complex character table of $A_5$, where $\varphi := \frac{1+\sqrt{5}}{2}$}
\label{complex chars of A_5}
\end{table}
Note that there are four rational characters: $\rho_1$, $\rho_{6} := \rho_3 + \widetilde{\rho_3}$, $\rho_4$ and $\rho_5$ (and the conjugacy classes $5A$ and $5B$ are indistinguishable by them). 

Let $p$ be a prime such that $p$ doesn't divide $|A_5| = 60$. 
By Theorem~\ref{Main_Result} and some standard arithmetic, the following conditions are equivalent:
\begin{itemize}
    \item $\mathbb{Z}_{(p)}A_5$ is semiperfect.
    \item The prime $p$ is inert in $\bQ (\varphi) = \bQ( \sqrt{5})$ (note that $\sO_{\bQ(\sqrt{5})}=\bZ \left[\varphi\right]$).
    \item $p \equiv 2,3\,(\text{mod } 5)$.
\end{itemize}

Next, consider $p=2,3,$ or $5$. We write ${}_p\phi_d$ for the Brauer character in characteristic $p$ of dimension $d$ (and $\widetilde{{}_p\phi_d}$ when there is a second distinct character of the same dimension -- there are never more than two). The Brauer character
for $\mathbb{F}_{2}A_{5}$, $\mathbb{F}_{3}A_{5}$, and
$\mathbb{F}_{5}A_{5}$ is shown in
Tables~\ref{tab2},\ref{tab3},\ref{tab4}.

\noindent
\begin{minipage}[t]{0.32\textwidth}
\centering
\captionsetup{type=table,justification=centering, width=\linewidth}
$$
\begin{array}{c|ccc}
    & 1 & 3 & 5 \\
\hline
{}_2\phi_1 & 1 & 1 & 1 \\
\widetilde{{}_2\phi_{4}} & 4 & -2 & -1 \\
{}_2\phi_{4} & 4 & 1 & -1 \\
\end{array}
$$
\caption{$\mathbb{F}_2A_5$}\label{tab2}
\end{minipage}%
\hfill
\begin{minipage}[t]{0.32\textwidth}
\centering
\captionsetup{type=table,justification=centering, width=\linewidth}
$$
\begin{array}{c|ccc}
    & 1 & 2 & 5 \\
\hline
{}_3\phi_1 & 1 & 1 & 1 \\
{}_3\phi_6 & 6 & -2 & 1 \\
{}_3\phi_4 & 4 & 0 & -1 \\
\end{array}
$$
\caption{$\mathbb{F}_3A_5$}\label{tab3}
\end{minipage}%
\hfill
\begin{minipage}[t]{0.32\textwidth}
\centering
\captionsetup{type=table,justification=centering, width=\linewidth}
$$
\begin{array}{c|ccc}
    & 1 & 2 & 3 \\
\hline
{}_5\phi_1 & 1 & 1 & 1 \\
{}_5\phi_3 & 3 & -1 & 0 \\
{}_5\phi_5 & 5 & 1 & -1 \\
\end{array}
$$
\caption{$\mathbb{F}_5A_5$}\label{tab4}
\end{minipage}
\\


We write $P(\phi)$ for the $\bF_p A_5$-projective cover of the simple $\mathbb{F}_pA_5$-module $\phi$ and $[P(\phi)]$ for its class in $\mathbf{G}(\bF_p A_5)$. Finally, given a rational character $\rho$, write $\text{Reduct}_p(\rho)$ for the ``reduction modulo $p$" of $\rho$. We claim that in all cases $\mathbb{Z}_{(p)}A_5$ is semiperfect. Indeed, semiperfectness of $\mathbb{Z}_{(p)}A_5$ boils down to whether the indecomposable projective modules of $\mathbb{F}_{(p)}A_5$ can be lifted. For each prime $p$, we explicitly give a lift (in the form of an idempotent in $\mathbb{Z}_{(p)}A_5$) for each indecomposable projective module $P({}_p\phi_d)$.

We start with $p=2$. Let us explicitly incorporate Conjecture~\ref{KGconj} into our discussion. We have the following reductions modulo $2$:
\[
\begin{aligned}
\text{Reduct}_2(\rho_1) = {}_{2}\phi_1 &\hspace{10ex} \text{Reduct}_2(\rho_6) = 2{}_{2}\phi_1 + \widetilde{{}_{2}\phi_4} \\
\text{Reduct}_2(\rho_4) = {}_{2}\phi_4 &\hspace{10ex} \text{Reduct}_2(\rho_5) = {}_{2}\phi_1 + \widetilde{{}_{2}\phi_4}
\end{aligned}
\]
Thus, $\text{Im}\big(\mathbf{G}^+(\bQ A_5) \xrightarrow{\alpha} \mathbf{G}^+(\bF_2 A_5)\big) = \text{span}_\mathbb{N}([{}_{2}\phi_1],[{}_{2}\phi_4], [{}_{2}\phi_1]+[\widetilde{{}_{2}\phi_4}])$. A short calculation (using the Cartan matrix or otherwise) yields that
\[
[P({}_{2}\phi_1)] = 4[{}_{2}\phi_1] + 2[\widetilde{{}_{2}\phi_4}], \  [P(\widetilde{{}_{2}\phi_4})] = 4[{}_{2}\phi_1] + 3[\widetilde{{}_{2}\phi_4}], \  [P({}_{2}\phi_4)] = [{}_{2}\phi_4].
\]
The group algebra $\mathbb{F}_2A_5$ has block decomposition $\mathbb{F}_2A_5 = B_0 \oplus B_1$, with:
\[B_0 = P({}_{2}\phi_1) \oplus P(\widetilde{{}_{2}\phi_4})^{\oplus 2},\hspace{5ex}B_1 = P({}_{2}\phi_4)^{\oplus 4} = {}_{2}\phi_4^{\oplus 4}\]
as left $\bF_2 A_5$-modules. 
The block $B_1$ is semisimple, and lifting $P({}_{2}\phi_4)$ is straightforward. 
We use Young symmetrisers to construct an idempotent in $\bQ S_5$, and modify it slightly to obtain a suitable idempotent in $\bZ_{(2)}A_5$. The strategy used to lift $P({}_{2}\phi_1)$ and $P(\widetilde{{}_{2}\phi_4})$ was different. We constructed idempotents in $\bF_2 A_5$, then lifted them $p$-adically.
Letting $\mathcal{L}(P(\phi)) \in \mathbb{Z}_{(2)}A_5$ denote an idempotent which ``lifts" $P(\phi)$, Table \ref{F2A5-idemopotents} describes the idempotents.

\begin{table}[h!]
    \fontsize{5.3pt}{5.3pt}\selectfont
    \centering
    \begin{tabular}[t]{|c|c|c|c|}
      \hline
      \phantom{\Big|}$A_5$ element & $\mathcal{L}(P({}_2\phi_1))$ & $\mathcal{L}(P(\widetilde{{}_2\phi_4}))$ & $\mathcal{L}(P({}_2\phi_4))$ \\
      \hline
    $()$& 1/5& 4/15& 1/15\\ 
    $(1,5,4,3,2)$& 1/5& -4/95 & 0\\ 
    $(1,4,2,5,3)$& 1/5& 4/19& 0\\ 
    $(1,3,5,2,4)$& 1/5& -98/285 & 0\\ 
    $(1,2,3,4,5)$& 1/5& -26/285& -1/15\\ 
    $(2,3,4)$& 0& -306/665& 1/15\\ 
    $(1,5,4)$& 0& -166/665 & 0\\ 
    $(1,4,5,3,2)$& 0& 146/1995 & 0\\ 
    $(1,3)$$(2,5)$& 0& -12/665 & 0\\ 
    $(1,2,4,3,5)$& 0& 496/1995 & -1/15\\ 
    $(2,5,4)$& 0& -12/665 & 0\\ 
    $(1,5,3,2,4)$& 0& 496/1995& 0\\ 
    $(1,4,5,2,3)$& 0& -306/665 & 0 \\
    $(1,3,5)$& 0& -166/665 & -1/15 \\ 
    $(1,2)$$(3,4)$& 0& 146/1995& 1/15\\ 
    $(2,4,3)$& 0& -88/665& 1/15\\ 
    $(1,5,4,2,3)$& 0& -103/1995& 0\\ 
    $(1,4)$$(3,5)$& 0& -8/105 & 0\\
    $(1,3,4,5,2)$& 0& -67/665 & 0\\ 
    $(1,2,5)$& 0& -8/399& -1/15\\ 
    $(3,5,4)$& 0& -22/95& 0\\ 
    $(1,5,3,4,2)$& 0& -1/57& 0\\ 
    $(1,4)$$(2,5)$& 0& 17/285& 0\\ 
    $(1,3,2,4,5)$& 0& 11/57& -1/15\\ 
    $(1,2,3)$& 0& 109/285& 1/15\\ 
    $(2,3,5)$& 0& -103/1995 & 0\\ 
    $(1,5)$$(3,4)$& 0& -8/105& -1/15\\ 
    $(1,4,2)$& 0& -67/665& 1/15\\ 
    $(1,3,2,5,4)$& 0& -8/399 & 0\\ 
    $(1,2,4,5,3)$& 0& -88/665& 0\\
      \hline
    \end{tabular}
    \hspace{1ex}
    \begin{tabular}[t]{|c|c|c|c|}
      \hline
      \phantom{\Big|}$A_5$ element & $\mathcal{L}(P({}_2\phi_1))$ & $\mathcal{L}(P(\widetilde{{}_2\phi_4}))$ & $\mathcal{L}(P({}_2\phi_4))$ \\
      \hline    
    $(2,5)$$(3,4)$& 0& 4/19& 0\\ 
    $(1,5)$$(2,4)$& 0& -98/285 & -1/15\\ 
    $(1,4)$$(2,3)$& 0& -26/285 & 1/15\\ 
    $(1,3)$$(4,5)$& 0& 4/15& 0\\ 
    $(1,2)$$(3,5)$& 0& -4/95& 0\\ 
    $(2,4)$$(3,5)$& 0& 8/57& 0\\ 
    $(1,5,2,3,4)$& 0& 27/95  & 0\\ 
    $(1,4,5)$& 0& 22/95 & -1/15\\ 
    $(1,3,2)$& 0& -6/19 & 1/15\\ 
    $(1,2,5,4,3)$& 0 & -17/285& 0\\ 
    $(3,4,5)$& 0& 118/1995& 0\\ 
    $(1,5,2)$& 0& 268/665& 0\\ 
    $(1,4,3,2,5)$& 0& -344/1995& -1/15\\ 
    $(1,3)$$(2,4)$& 0& -211/1995& 1/15\\ 
    $(1,2,3,5,4)$& 0& -127/1995 & 0\\ 
    $(2,3)$$(4,5)$& 0& 109/285& 0\\ 
    $(1,5,3)$& 0 & -22/95 & 0\\ 
    $(1,4,3,5,2)$& 0& -1/57& 0\\ 
    $(1,3,4,2,5)$& 0& 17/285& -1/15\\ 
    $(1,2,4)$& 0 & 11/57& 1/15\\ 
    $(2,5,3)$& 0& -17/285& 0\\ 
    $(1,5,2,4,3)$& 0& 8/57& 0\\ 
    $(1,4,2,3,5)$& 0& 27/95& -1/15\\ 
    $(1,3,4)$& 0& 22/95& 1/15\\ 
    $(1,2)$$(4,5)$& 0& -6/19 & 0\\ 
    $(2,4,5)$& 0 & -211/1995& 0\\ 
    $(1,5)$$(2,3)$& 0& -127/1995& -1/15\\ 
    $(1,4,3)$& 0& 118/1995& 1/15\\ 
    $(1,3,5,4,2)$& 0& 268/665& 0\\
    $(1,2,5,3,4)$& 0& -344/1995& 0\\
      \hline
    \end{tabular}
    \caption{Coefficients of idempotents in $\mathbb{Z}_{(2)}A_5$ which lift projective indecomposable modules of $\mathbb{F}_2A_5$}
    \label{F2A5-idemopotents}
\end{table}

For $p=3$ and $p=5$, the conditions of Conjecture~\ref{KGconj} also hold:
\[
\begin{aligned}
\text{Reduct}_3(\rho_1) = {}_{3}\phi_1 &\hspace{10ex} \text{Reduct}_3(\rho_6) = {}_{3}\phi_6 \\
\text{Reduct}_3(\rho_4) = {}_{3}\phi_4 &\hspace{10ex} \text{Reduct}_3(\rho_5) = {}_{3}\phi_1 + {{}_{3}\phi_4}
\end{aligned}
\]
Thus, $\text{Im}\big(\mathbf{G}^+(\bQ A_5) \xrightarrow{\alpha} \mathbf{G}^+(\bF_3 A_5)\big) = \text{span}_\mathbb{N}([{}_{3}\phi_1],[{}_{3}\phi_4],[{}_{3}\phi_6])$ and
\[
[P({}_{3}\phi_1)] = 2[{}_{3}\phi_1] + [{{}_{3}\phi_4}], \  [P({{}_{3}\phi_4})] = [{}_{3}\phi_1] + 2[{{}_{3}\phi_4}], \  [P({}_{3}\phi_6)] = [{}_{3}\phi_6].
\]
Moreover, we have the following reductions modulo $5$:
\[
\begin{aligned}
\text{Reduct}_5(\rho_1) = {}_{5}\phi_1 &\hspace{10ex} \text{Reduct}_5(\rho_6) = 2{}_{5}\phi_3 \\
\text{Reduct}_5(\rho_4) = {}_{5}\phi_1 + {}_{5}\phi_3 &\hspace{10ex} \text{Reduct}_5(\rho_5) = {}_{5}\phi_5
\end{aligned}
\]
Thus, $\text{Im}\big(\mathbf{G}^+(\bQ A_5) \xrightarrow{\alpha} \mathbf{G}^+(\bF_5 A_5)\big) = \text{span}_\mathbb{N}([{}_{5}\phi_1],2[{}_{5}\phi_3], [{}_{5}\phi_1]+[{{}_{5}\phi_3}],[{}_{5}\phi_5])$ and 
\[
[P({}_{5}\phi_1)] = 2[{}_{5}\phi_1] + [{{}_{5}\phi_3}], \  [P({{}_{5}\phi_3})] = [{}_{5}\phi_1] + 3[{{}_{5}\phi_3}], \  [P({}_{5}\phi_5)] = [{}_{5}\phi_5].
\]
The strategies to construct idempotents were similar to the $p=2$ case. 
The code and further details are deposited on GitHub \cite{Jo}. Once again,  $\mathcal{L}(P({}_p\phi_d)) \in \mathbb{Z}_{(p)}A_5$ denotes an idempotent which ``lifts" $P({}_p\phi_d)$. The tables Table \ref{F3A5-idemopotents} and Table \ref{F5A5-idemopotents} describe the idempotents in $\mathbb{Z}_{(3)}A_5$ and $\mathbb{Z}_{(5)}A_5$, respectively.

\begin{table}[h!]
    \fontsize{5.3pt}{5.3pt}\selectfont
    \centering
    \begin{tabular}[t]{|c|c|c|c|}
      \hline 
      \phantom{\Big|}$A_5$ element & $\mathcal{L}(P({}_3\phi_1))$ & $\mathcal{L}(P({}_3\phi_6))$ & $\mathcal{L}(P({}_3\phi_4))$ \\
      \hline
    $()$& 1/10& 1/10 & 3/20\\ 
    $(1,5,4,3,2)$ & 3/80& 0 & -3/80\\ 
    $(1,4,2,5,3)$ & 3/80& 0 & -3/80\\ 
    $(1,3,5,2,4)$ & 3/80& 0 & -3/80\\ 
    $(1,2,3,4,5)$ & 3/80& 1/10 & -3/80\\ 
    $(2,3,4)$& -1/40 & 0 & 1/40\\ 
    $(1,5,4)$& -1/40& 1/10 & 1/40\\ 
    $(1,4,5,3,2)$& 3/80& 0 & -3/80\\ 
    $(1,3)$$(2,5)$& 3/80& 0 & -3/80\\ 
    $(1,2,4,3,5)$& -1/40& 0 & 1/40\\ 
    $(2,5,4)$& 3/80& 0& -3/80\\ 
    $(1,5,3,2,4)$& 3/80& 0& -3/80\\ 
    $(1,4,5,2,3)$ & 3/80& 0& -3/80\\
    $(1,3,5)$& 3/80& -1/10& -3/80\\ 
    $(1,2)$$(3,4)$& 1/10& 0 & 3/20\\ 
    $(2,4,3)$& -1/40& 0  & 1/40\\ 
    $(1,5,4,2,3)$& 3/80& 0 & -3/80\\ 
    $(1,4)$$(3,5)$& -1/40& 0  & 1/40\\
    $(1,3,4,5,2)$& -1/40& 0  & 1/40\\ 
    $(1,2,5)$& 3/80& -1/10  & -3/80\\ 
    $(3,5,4)$& 3/80& 0  & -3/80\\ 
    $(1,5,3,4,2)$& -1/40& 0  & 1/40\\ 
    $(1,4)$$(2,5)$& -1/40& 0  & 1/40\\ 
    $(1,3,2,4,5)$& 3/80& 1/10 & -3/80\\ 
    $(1,2,3)$ & -1/40& 1/10 & 1/40\\ 
    $(2,3,5)$ & -1/40& 0 & 1/40\\ 
    $(1,5)$$(3,4)$& 3/80& 0& -3/80\\ 
    $(1,4,2)$& -1/40& 0& 1/40\\ 
    $(1,3,2,5,4)$& 3/80& 1/10   & -3/80\\ 
    $(1,2,4,5,3)$& -1/40& 0  & 1/40\\
      \hline
    \end{tabular}
    \hspace{1ex}
    \begin{tabular}[t]{|c|c|c|c|}
      \hline
      \phantom{\Big|}$A_5$ element & $\mathcal{L}(P({}_3\phi_1))$ & $\mathcal{L}(P({}_3\phi_6))$ & $\mathcal{L}(P({}_3\phi_4))$ \\
      \hline    
    $(2,5)$$(3,4)$& 3/80& 0  & -3/80\\ 
    $(1,5)$$(2,4)$& 3/80& 0 & -3/80\\ 
    $(1,4)$$(2,3)$& 1/10& -1/10 & 3/20\\ 
    $(1,3)$$(4,5)$& 3/80& -1/10 & -3/80\\ 
    $(1,2)$$(3,5)$& 3/80& 0 & -3/80\\ 
    $(2,4)$$(3,5)$& 3/80& 0 & -3/80\\ 
    $(1,5,2,3,4)$& 3/80& 0 & -3/80\\ 
    $(1,4,5)$& -1/40& 1/10 & 1/40\\ 
    $(1,3,2)$& -1/40& 1/10 & 1/40\\ 
    $(1,2,5,4,3)$& -1/40& 0 & 1/40\\ 
    $(3,4,5)$ & 3/80& 0& -3/80\\ 
    $(1,5,2)$ & 3/80& 0& -3/80\\ 
    $(1,4,3,2,5)$& 3/80& 0& -3/80\\ 
    $(1,3)$$(2,4)$& 1/10& 0& 3/20\\ 
    $(1,2,3,5,4)$& 3/80& 1/10& -3/80\\ 
    $(2,3)$$(4,5)$& -1/40& -1/10& 1/40\\ 
    $(1,5,3)$& 3/80& 0 & -3/80\\ 
    $(1,4,3,5,2)$& 3/80& 0  & -3/80\\ 
    $(1,3,4,2,5)$& -1/40& 0  & 1/40\\ 
    $(1,2,4)$& -1/40 & -1/10  & 1/40\\ 
    $(2,5,3)$& -1/40& 0  & 1/40\\ 
    $(1,5,2,4,3)$& -1/40& 0  & 1/40\\ 
    $(1,4,2,3,5)$& 3/80& 0 & -3/80\\ 
    $(1,3,4)$& -1/40& -1/10 & 1/40\\ 
    $(1,2)$$(4,5)$& 3/80& -1/10 & -3/80\\ 
    $(2,4,5)$& 3/80& 0 & -3/80\\ 
    $(1,5)$$(2,3)$& -1/40& -1/10  & 1/40\\ 
    $(1,4,3)$& -1/40& 0  & 1/40\\ 
    $(1,3,5,4,2)$& -1/40& 0  & 1/40\\
    $(1,2,5,3,4)$ & 3/80& 0  & -3/80\\
      \hline
    \end{tabular}
    \caption{Coefficients of idempotents in $\mathbb{Z}_{(3)}A_5$ which lift projective indecomposable modules of $\mathbb{F}_3A_5$}
    \label{F3A5-idemopotents}
\end{table}

\begin{table}[h!]
    \fontsize{5.3pt}{5.3pt}\selectfont
    \centering
    \begin{tabular}[t]{|c|c|c|c|}
      \hline 
      \phantom{\Big|}$A_5$ element & $\mathcal{L}(P({}_5\phi_1))$ & $\mathcal{L}(P({}_5\phi_3))$ & $\mathcal{L}(P({}_5\phi_5))$ \\
      \hline
    $()$ & 1/12 & 1/6& 1/12\\ 
    $(1,5,4,3,2)$   & 0 & 0 & 0\\ 
    $(1,4,2,5,3)$   & 0 & 0& -1/12\\ 
    $(1,3,5,2,4)$   & 0 & 0& 0\\ 
    $(1,2,3,4,5)$   & 0 & 0 & 1/12\\ 
    $(2,3,4)$  & 0 & 0 & 0\\ 
    $(1,5,4)$ & 1/12 & 0 & 0\\ 
    $(1,4,5,3,2)$  & 0 & 0 & 1/12\\ 
    $(1,3)$$(2,5)$  & 0 & 0  & 0\\ 
    $(1,2,4,3,5)$  & 0 & 0  & 0\\ 
    $(2,5,4)$  & 0 & 0  & 0\\ 
    $(1,5,3,2,4)$  & 0 & 0  & -1/12\\ 
    $(1,4,5,2,3)$  & 0  & 0 & 0\\
    $(1,3,5)$ & 1/12 & 0  & 1/12\\ 
    $(1,2)$$(3,4)$  & 0  & -1/6  & 1/12\\ 
    $(2,4,3)$  & 0 & 0 & -1/12\\ 
    $(1,5,4,2,3)$  & 0 & 0 & 0\\ 
    $(1,4)$$(3,5)$  & 1/12 & 0 & 0 \\
    $(1,3,4,5,2)$ & 0 & 0 & 0\\ 
    $(1,2,5)$ & 0 & 0  & -1/12\\ 
    $(3,5,4)$& 1/12 & 1/6 & 0\\ 
    $(1,5,3,4,2)$  & 0 & 0 & 0\\ 
    $(1,4)$$(2,5)$  & 0 & 0 & 0\\ 
    $(1,3,2,4,5)$  & 0 & 0 & -1/12\\ 
    $(1,2,3)$  & 0 & 0 & -1/12\\ 
    $(2,3,5)$  & 0 & 0  & 0\\ 
    $(1,5)$$(3,4)$ & 1/12 & 0 & -1/12\\ 
    $(1,4,2)$ & 0 & 0 & -1/12\\ 
    $(1,3,2,5,4)$ & 0 & 0 & 0\\ 
    $(1,2,4,5,3)$ & 0 & 0 & 0\\
      \hline
    \end{tabular}
    \hspace{1ex}
    \begin{tabular}[t]{|c|c|c|c|}
      \hline
      \phantom{\Big|}$A_5$ element & $\mathcal{L}(P({}_5\phi_1))$ & $\mathcal{L}(P({}_5\phi_3))$ & $\mathcal{L}(P({}_5\phi_5))$ \\
      \hline    
    $(2,5)$$(3,4)$ & 0 & 0 & 0\\ 
    $(1,5)$$(2,4)$ & 0 & 0 & 1/12\\ 
    $(1,4)$$(2,3)$ & 0 & 0 & 1/12\\ 
    $(1,3)$$(4,5)$ & 1/12 & 0  & 0\\ 
    $(1,2)$$(3,5)$ & 0  & -1/6& -1/12\\ 
    $(2,4)$$(3,5)$ & 0 & 0 & 1/12\\ 
    $(1,5,2,3,4)$ & 0 & 0 & 0\\ 
    $(1,4,5)$& 1/12 & 0 & 0\\ 
    $(1,3,2)$  & 0 & 0 & 0\\ 
    $(1,2,5,4,3)$  & 0 & 0 & 0\\ 
    $(3,4,5)$ & 1/12  & 1/6  & -1/12\\ 
    $(1,5,2)$  & 0 & 0  & 0\\ 
    $(1,4,3,2,5)$  & 0 & 0 & 1/12\\ 
    $(1,3)$$(2,4)$  & 0 & 0 & 1/12\\ 
    $(1,2,3,5,4)$  & 0 & 0  & 0\\ 
    $(2,3)$$(4,5)$  & 0 & 0  & 0\\ 
    $(1,5,3)$   & 1/12 & 0 & 1/12\\ 
    $(1,4,3,5,2)$ & 0 & 0 & 0\\ 
    $(1,3,4,2,5)$ & 0 & 0 & 0\\ 
    $(1,2,4)$ & 0 & 0 & 0\\ 
    $(2,5,3)$ & 0 & 0 & 0\\ 
    $(1,5,2,4,3)$ & 0 & 0 & 0\\ 
    $(1,4,2,3,5)$ & 0 & 0 & -1/12\\ 
    $(1,3,4)$  & 1/12 & 0 & -1/12\\ 
    $(1,2)$$(4,5)$   & 0  & -1/6  & 0\\ 
    $(2,4,5)$   & 0 & 0  & 0\\ 
    $(1,5)$$(2,3)$   & 0 & 0  & 0\\ 
    $(1,4,3)$  & 1/12 & 0  & 0\\ 
    $(1,3,5,4,2)$  & 0 & 0  & 0\\
    $(1,2,5,3,4)$  & 0 & 0 & 1/12\\
      \hline
    \end{tabular}
    \caption{Coefficients of idempotents in $\mathbb{Z}_{(5)}A_5$ which lift projective indecomposable modules of $\mathbb{F}_5A_5$}
    \label{F5A5-idemopotents}
\end{table}


\end{ex}

\newpage 
\bibstyle{numeric}
\printbibliography

\end{document}